\documentclass[10pt]{amsart}
\usepackage{amsmath}
\usepackage{amsfonts}
\usepackage{amssymb}

\usepackage{graphicx, color}

\newcommand{\Diltwo}[1]{\Dilat_{#1}}
\def\Dilat{\operatorname{D}}

\newcommand{\Bignorm}[1]{\Bigl\|#1\Bigr\|}
\newcommand{\bignorm}[1]{\bigl\lVert#1\bigr\rVert}

\newcommand{\wholenorm}[1]{\left\|#1\right\|}

\newcommand{\bC}{{\mathbb C}}
\newcommand{\Rst}{{\mathbb R}}

\newcommand{\Rdst}{{\Rst^d}}

\newcommand{\set}[2]{\big\{ \, #1 \, \big| \, #2 \, \big\}}
\newcommand{\norm}[1]{\lVert#1\rVert}
\newcommand{\Esp}{{\boldsymbol E}}

\newcommand{\Zst}{{\mathbb Z}}
\newcommand{\Zdst}{{\Zst^d}}

\newcommand{\supp}{\operatorname{supp}}
\newcommand{\Lsp}{{\boldsymbol L}}
\newcommand{\Ltsp}{{\Lsp^2}}
\newcommand{\LtRd}{{\Ltsp(\Rst^d)}}
\newcommand{\Nst}{{\mathbb N}}

\def\Esp{{\boldsymbol E}}

\newcommand{\finiteset}{X}

\newcommand{\ip}[2]{\ensuremath{\left<#1,#2\right>}}
\newcommand{\sett}[1]{\ensuremath{\left \{ #1 \right \}}}
\newcommand{\abs}[1]{\ensuremath{\left| #1 \right| }}
\newcommand{\w}[2]{\left(1+\abs{#1}\right)^{#2}}

\newcommand{\Schwartz}{\mathcal{S}}
\newcommand{\bes}{{\dot{B}^{\alpha, q}_{p}}}

\newcommand{\besseq}{{\dot{b}^{\alpha, q}_{p}}}
\newcommand{\trie}{{\dot{F}^{\alpha, q}_{p}}}

\newcommand{\trieseq}{{\dot{f}^{\alpha, q}_{p}}}

\newcommand{\essp}{{\mathbf e^{\alpha, q}_{p}}}
\newcommand{\Essp}{{\mathbf E^{\alpha, q}_{p}}}
\newcommand{\Esspdual}{{\mathbf E^{-\alpha, q'}_{p'}}}
\newcommand{\EsspHil}{{\mathbf E^{0, 2}_{2}}}
\newcommand{\esspz}{{\mathbf e^{\alpha_0, q_0}_{p_0}}}
\newcommand{\Esspz}{{\mathbf E^{\alpha_0, q_0}_{p_0}}}
\newcommand{\esspo}{{\mathbf e^{\alpha_1, q_1}_{p_1}}}
\newcommand{\Esspo}{{\mathbf E^{\alpha_1, q_1}_{p_1}}}

\newcommand{\varpsi}{\tau}

\newcommand{\tilpsil}{{\widetilde{\psi_\lambda}}}
\newcommand{\tilpsijl}{{\widetilde{\psi}_{j,\lambda}}}
\newcommand{\psijl}{\psi_{j,\lambda}}
\newcommand{\coefpsi}{{C_\psi}}

\newcommand{\recvarpsi}{{S_{\varpsi}}}
\newcommand{\coefPsi}{{C_\Psi}}
\newcommand{\recPsi}{{S_\Psi}}
\newcommand{\tilPsi}{{\widetilde{\Psi}}}
\newcommand{\coeftilPsi}{{C_\tilPsi}}
\newcommand{\rectilPsi}{{S_\tilPsi}}

\newcommand{\Phieps}{\Phi^\varepsilon}

\newcommand{\tilphijl}{{\widetilde{\phi}_{j,\lambda}}}
\newcommand{\phijl}{\phi_{j,\lambda}}

\newcommand{\coefPhi}{{C_\Phi}}
\newcommand{\recPhi}{{S_\Phi}}
\newcommand{\tilPhi}{{\widetilde{\Phi}}}
\newcommand{\coeftilPhi}{{C_\tilPhi}}

\newcommand{\phieps}{\varphi^\varepsilon}
\newcommand{\phiepsjl}{\varphi^\varepsilon_{j,\lambda}}

\newcommand{\coefPhieps}{{C_{\Phi^\varepsilon}}}
\newcommand{\recPhieps}{{S_{\Phi^\varepsilon}}}

\newcommand{\quasiA}{p_A}
\newcommand{\alphanot}{\beta}
\newcommand{\consexp}{c}
\newcommand{\wavclass}{\ensuremath{\mathcal{C}(A,\Lambda)}}
\newcommand{\supereq}{=}

\newtheorem{deff}{Definition}
\newtheorem{theo}{Theorem}
\newtheorem*{theo*}{Theorem}

\newtheorem{lemma}{Lemma}
\newtheorem{coro}{Corollary}
\newtheorem{remark}{Remark}

\title[Non-uniform painless decompositions]{Non-uniform painless
decompositions for anisotropic Besov and Triebel-Lizorkin spaces}

\author[C. Cabrelli]{Carlos Cabrelli}
\address{Departamento de
Matem\'atica \\ Facultad de Ciencias Exactas y Naturales\\ Universidad
de Buenos Aires\\ Ciudad Universitaria, Pabell\'on I\\ 1428 Capital
Federal\\ ARGENTINA\\ and IMAS, UBA-CONICET, Argentina}
\email[Carlos~Cabrelli]{cabrelli@dm.uba.ar}

\author[U.Molter]{Ursula Molter}
\email[Ursula~M.~Molter]{umolter@dm.uba.ar}

\author[J.~L.~Romero]{Jos\'e Luis Romero}
\email[Jos\'e Luis Romero]{jlromero@dm.uba.ar}

\thanks{The authors acknowledge
support from the following grants: 
PICT 2006-00177 (ANPCyT), and PIP 2008-398 (CONICET)
and UBACyT X149 and X028 (UBA)}


\begin{document}
\begin{abstract}
In this article we construct affine systems that provide a simultaneous
atomic decomposition for a wide class of functional spaces including the
Lebesgue spaces $L^p(\Rdst)$, $1<p<+\infty$. The novelty and difficulty of this
construction is that we allow for non-lattice translations.

We prove that for an arbitrary expansive matrix $A$ and any set $\Lambda$ -
satisfying a certain spreadness condition but otherwise irregular- there exists
a smooth window whose translations along the elements of $\Lambda$ and dilations
by powers of $A$ provide an atomic decomposition for the whole range of the
anisotropic Triebel-Lizorkin spaces. The generating window can be either chosen
to be bandlimited or to have compact support.

To derive these results we start with a known general ``painless'' construction
that has recently appeared in the literature. We show that this construction
extends to Besov and Triebel-Lizorkin spaces by providing adequate dual
systems.
\end{abstract}
\maketitle
\section{Introduction}
The membership of a distribution to a large number of classical functional
spaces can be characterized by its Littlewood-Paley decomposition,
which consists of a sequence of smooth frequency cut-offs at dyadic scales.
The functional spaces that can be described in that way are
generically known as Besov and Triebel-Lizorkin spaces. This class includes
among others the Lebesgue spaces $L^p(\Rdst), (1<p<+\infty)$, Sobolev spaces
and Lipschitz spaces.

More recently, anisotropic variants of these spaces have been introduced,
where the dyadic scales are replaced by more general ones allowing
different spatial directions to be dilated by different factors (see
\cite{lotr79, sctr73, gata02, tr04-2, bo05-1, boho06, bo07b, bo08-3}
and the references therein). These variants are useful for example to study
anisotropic smoothness conditions.

Time-scale atomic decompositions are a very powerful tool to analyze Besov and
Triebel-Lizorkin spaces. The technique consists of representing a general
distribution $f \in \mathcal{S}'(\Rdst)$ as a superimposition of
atoms $\sett{\psi_{j,k}: j\in\Zst, k\in \Zdst}$,
\begin{align*}
f = \sum_{j \in \Zst} \sum_{k \in \Zdst} c_{j,k} \psi_{j,k}.
\end{align*}
The membership of $f$ to a particular Besov or Triebel-Lizorkin space is
characterized by the decay of its coefficients $c_{j,k}$. In the classical
(isotropic) case, the atoms are of the form
$\psi_{j,k}(x)=2^{-j/2}\psi(2^{-j}x-k)$, where $\psi$ is an adequate window
function called a wavelet. In the anisotropic case a number of
alternatives are possible. One of them is to replace the powers of 2 by
powers of a more general matrix $A$, yielding atoms of the form,
\begin{align*}
\set{\abs{\det(A)}^{-j/2} \psi(A^{-j} \cdot - k)}
{j \in \Zst, k \in \Zdst}.
\end{align*}
(See \cite{tr04-2} for other alternatives.)
Existence of atomic decompositions for anisotropic Besov and Triebel-Lizorkin
spaces is a well-known fact. There is an ample literature giving sets of
atoms with specific properties (see \cite{tr04-2} and the references therein).
The purpose of this article is to show
that these spaces also admit atomic decompositions where the
integer translations are replaced by translations along quite arbitrary sets.
Given a matrix $A \in \Rst^{d \times d}$ that is \emph{expansive} (i.e. all its
eigenvalues $\mu$ satisfy $\abs{\mu}>1$) and a set $\Lambda \subseteq \Rdst$
that satisfies a certain spreadness condition but is
otherwise irregular, we show that there exists a function $\psi \in
\mathcal{S}$ such that the irregular time-scale system,
\begin{align}
\label{eq_intro_irrwav}
W(\psi, A, \Lambda) :=
\set{\abs{\det(A)}^{-j/2} \psi(A^{-j} \cdot - \lambda)}
{j \in \Zst, \lambda \in \Lambda},
\end{align}
gives an atomic decomposition of the whole class of (anisotropic) Besov and
Triebel-Lizorkin spaces. The function $\psi$ can be chosen to be bandlimited
or to have compact support (in this latter case, since $\psi$ cannot have
infinitely many vanishing moments, we have to restrict the
decomposition to a subclass of spaces having a bounded degree of smoothness).

This result is new even in the isotropic case. Its relevance stems from the
fact that the set of translation nodes $\Lambda$ is not assumed to
have any kind of regularity (besides a mild spreadness condition). As a comparison to our work,
the only results giving such irregular time-scale decompositions that we are
aware of are: the one that comes from the general theory of atomic
decompositions of coorbit spaces associated with group
representations \cite{fegr89}, and the one obtained from oversampling Calder\'on's continuous
resolution of the identity \cite{gihaholawewe02, lisu12}.
These prove the existence of time-scale atomic
decompositions using irregular sets of translates, \emph{as long as they are
sufficiently dense} (in a sense that may be hard to quantify). In contrast,
our result proves that \emph{any} set of translates can be used (under a
mild spreadness assumption).

Moreover, the function $\psi$ is explicitly constructed following a very
concrete method. In fact, the starting point of the article is a
construction in \cite{alcamo04-1} that can be regarded as a generalization of
the so-called painless method of Daubechies, Grossmann, and Meyer
\cite{dagrme86}  - which we consequently call the \emph{generalized painless
method}. It mainly consists of using a number of geometric barriers
to ensure the validity of the so-called \emph{frame inequality},
\begin{align*}
A \norm{f}_{L^2}^2 \leq
\sum_{j \in \Zst, \lambda \in \Lambda}
\abs{\ip{f}{\abs{\det(A)}^{-j/2} \psi(A^{-j} \cdot - \lambda)}}^2
\leq B \norm{f}_{L^2}^2.
\end{align*}
This is enough to deduce that the atoms in Eq. \eqref{eq_intro_irrwav}
give an atomic decomposition of $L^2(\Rdst)$. This method of proof, however,
does not provide an explicit \emph{dual system} for those atoms,
that is, an explicit family of coefficient functional $c_{j,k}$ such that,
\begin{align}
\label{eq_intro_irrwav_exp}
f = \sum_{j,\lambda} c_{j,\lambda} \abs{\det(A)}^{-j/2}
\psi(A^{-j} \cdot - \lambda).
\end{align}
The existence of such a dual system can be deduced by Hilbert space
arguments, but not having explicitly produced them it is not clear whether
the expansion in Eq. \eqref{eq_intro_irrwav_exp} (valid in $L^2$)
extends to an atomic decomposition of all (anisotropic) Besov and
Triebel-Lizorkin spaces. This is not a trivial question. In sheer contrast to 
the case of time-frequency
analysis \cite{gr04-1,grle04,bacahela06,bacahela06-1,fogr05}, it is known
that there exist ``nice'' (isotropic) time-scale systems satisfying the frame
inequality but failing to give an atomic decomposition for a certain range of
$L^p$ spaces $(1<p<+\infty)$ \cite{tch87,tch89, come00} (see also
\cite{bula09, bula09-2}).

We will show that every wavelet system constructed following the
``generalized painless method'' yields an atomic decomposition of the
class of anisotropic Besov-Triebel-Lizorkin spaces. For atoms constructed
following that method, each scale interacts with a bounded number of other
scales. Even having this property, the fact that a time-scale system
satisfies the frame inequality does not always imply that it gives an
atomic decomposition of Besov and Triebel-Lizorkin spaces. To prove the
existence of well-behaved dual systems we need to resort to certain results from
A. Beurling on the balayage problem \cite{be66-3}.

The generating windows of these constructions are band-limited. However, often it is desirable to have generators that are
compactly supported in time. We devote the last section of this article to show how to adapt the previous results to obtain
atomic decompositions with generators that are smooth, compactly supported and with an arbitrary number of vanishing moments.

\subsection{Overview of the results}
The following is the main result of the article.
\begin{theo*}
Let $\Lambda \subseteq \Rdst$ be a set that is well-spread, that is,
\begin{align*}
& \sup_{x \in \Rdst} \#(\Lambda \cap ([-1/2,1/2]^d+\sett{x})) < +\infty,
\\
& \bigcup_{\lambda \in \Lambda} B_R(\lambda) = \Rdst,
\quad \mbox{for some }R>0.
\end{align*}
Let $A \in \Rst^{d\times d}$ be an expansive matrix (i.e. all its eigenvalues $\mu$ satisfy $\abs{\mu}>1$).

Then there exists a Schwartz class function $\psi$ with Fourier transform supported
on a compact set not containing the origin, such that the irregular wavelet
system,
\begin{align*}
W(\psi, A, \Lambda) = \set{\psijl := \abs{\det(A)}^{-j/2} \psi(A^{-j} \cdot - \lambda)}{j\in\Zst, \lambda \in \Lambda},
\end{align*}
provides an atomic decomposition for the whole class of anisotropic Besov and Triebel-Lizorkin spaces.
\end{theo*}
Alternatively, $\psi$ can be chosen to have compact support and an arbitrary large number of vanishing moments, but in this
case the decomposition is restricted to a subclass of spaces having a bounded degree of smoothness. The meaning of the term
``atomic decomposition''
is carefully discussed below (see Theorems \ref{th_bandlimited} and \ref{th_compact}). It implies that there is an expansion
$f = \sum_{j,\lambda} c_{j,\lambda} \psijl$ valid in all Besov-Triebel-Lizorkin spaces and that the norm of $f$ in those
spaces is equivalent to a suitable norm depending on the absolute value of the coefficients $c_{j, \lambda}$.

In addition to the oversampling results that provide irregular atomic decompositions when $\Lambda$ is sufficiently dense
\cite{fegr89, gihaholawewe02, lisu12}, we mention as related literature the work on wavelets sets (see \cite{dalasp97,wa02-2}
and the references therein), where the existence of windows $\psi$ such that the time-scale system $W(\psi, A, \Lambda)$ is
an orthonormal basis of $\LtRd$ is proved. In that work the Fourier transform of $\psi$ is the characteristic function of a
set $\Omega$ that is related to $\Lambda$ by a tiling condition. This requires $\Lambda$ to have a certain regularity (e.g.
to be a translation of a subgroup of $\Rdst$). When $\Lambda$ is a lattice and the the matrix $A$ is further assumed to
preserve $\Lambda$ - i.e. $A \Lambda \subseteq \Lambda$ - even fast-decaying wavelets are known to exist \cite{bo01-7,
bosp02}.

We stress that the window $\psi$ is not only proved to exist, but it is
obtained following a concrete construction that we call the ``generalized
painless method'' (see Section \ref{sec:cons}). We prove that the
irregular time-scale system thus constructed admits a suitable dual frame.

Besides the theoretical relevance of our result, frames with arbitrary
translation nodes can be useful in applications where data is modeled by
irregular point sources, for example when it is acquired from irregularly
distributed measurements. In order for our result to have full practical
relevance the computability of the dual system should be investigated. We do
not pursue this in this article
(see Section \ref{sec:generalization}). However, the main result has the following consequence of practical interest: if
$\psi$ is constructed following the recipe in Section \ref{sec:cons}, then the membership of a function $f$ to all
anisotropic Besov-Triebel-Lizorkin spaces can be characterized
in terms of the \emph{size} of the numbers $\ip{f}{\psijl}$, which are computable. This means for example that
\emph{asymptotic} information about the smoothness and $p$-integrability of a function $f$ can be extracted from the
size-profile of the numbers $\ip{f}{\psijl}$.

\subsection{Technical overview}
The main technical point of this article is to combine the $\phi$-transform methods of Frazier and Jawerth - and their recent
extension to the anisotropic setting - with Beurling's work on the balayage problem.

The $\phi$-transform theory produces
bandlimited windows that generate time-scale atomic decompositions of the class of Besov-Triebel-Lizorkin spaces using
translates along a lattice. 

The balayage problem consists in deciding, given a compact set $K \subseteq \Rdst$ and a (possibly discrete) set $\Lambda
\subseteq \Rdst$, whether the Fourier transform of every measure coincides on $K$ with the Fourier transform of some measure
supported on $\Lambda$. Beurling's fundamental work gives sufficient conditions for this, relating the size of $K$ to the
density of $\Lambda$.

The main result of this article is obtained by combining both theories in the following way. Given an expansive matrix $A$,
a well-spread set $\Lambda$ and a window constructed following the generalized painless method, we use the $\phi$-transform
theory to produce an atomic decomposition of
Besov-Triebel-Lizorkin spaces using powers of the matrix $A$ as dilations and a lattice $\Gamma$ as set of translation nodes.
The density of $\Gamma$ is chosen so as to roughly match the density of $\Lambda$; this is possible thanks to the flexibility
of the $\phi$-transform method. We then use Beurling's results to
represent every measure supported on $\Gamma$ as a superimposition of delta measures supported on $\Lambda$. This allows us
to replace the lattice translations of the $\phi$-transform expansion by translations along $\Lambda$. Finally, a careful
perturbation argument yields a compactly supported window.

The $\phi$-transform theory provides us with a number of tools to tackle the technical problems. We use the recent extension
of these tools to the anisotropic setting (see Sections \ref{sec:spaces} and \ref{sec:tech}).
\subsection{Organization}
The article is organized as follows. Section \ref{sec:cons} presents the
generalized painless method. This is known to produce wavelet frames for
$L^2(\Rdst)$. Sections \ref{sec:spaces} and \ref{sec:bal} introduce
the relevant classes of functional spaces and collect several technical
tools. In Section \ref{sec:gen} we prove that windows constructed following
the generalized painless method provide atomic decompositions for Besov and
Triebel Lizorkin spaces. As a consequence, these spaces admit an atomic
decomposition produced by arbitrary translations and dilations of a
bandlimited function. In Section \ref{sec:compact} we study the case of 
compactly supported generators. Finally, Section \ref{sec:generalization} discusses
possible generalizations and extensions of the results.

\section{The generalized painless method in $L^2$}
\label{sec:cons}

A matrix $A \in \Rst^{d \times d}$ is called \emph{expansive} if
all its eigenvalues $\mu$ satisfy $\abs{\mu}>1$. For a general
matrix $A \in {\bC}^{d \times d}$, we denote by $A^t$ its
transpose matrix whereas $A^*$ will denote its conjugate transpose.
A set $\Lambda \subseteq \Rdst$ is called \emph{relatively separated} if,
\begin{align*}
\sup_{x \in \Rdst} \#(\Lambda \cap ([-1/2,1/2]^d+\sett{x})) < +\infty.
\end{align*}
That is, a set is relatively separated if the number of points it has on any
cube of side-length 1 is bounded. Equivalently, a set $\Lambda$ is relatively
separated if it can be split into a finite union of sets
$\Lambda^1, \ldots, \Lambda^n$ with each set $\Lambda^i$ being
\emph{separated},
that is,
\begin{align*}
\inf \sett{\abs{\lambda-\lambda'} : \lambda, \lambda' \in \Lambda^i,
\lambda \not= \lambda'} >0.
\end{align*}

The \emph{gap} of a set $\Lambda \subseteq \Rdst$ is defined as,
\begin{align}
\label{eq_gap}
\rho(\Lambda) := \sup_{x\in\Rdst} \inf_{\lambda \in \Lambda} \abs{x-\lambda},
\end{align}
where $\abs{\cdot}$ denotes the Euclidean norm.
A set $\Lambda$ is called \emph{relatively dense}
if $\rho(\Lambda) < + \infty$. Equivalently, $\Lambda$ is relatively dense
if there exists $R>0$ such that,
\begin{align*}
\Rdst = \bigcup_{\lambda \in \Lambda} B_R(\lambda),
\end{align*}
where, in general, $B_r(x)$ denotes the open Euclidean ball with center $x$ and
radius $r$.
For $x \in \Rdst$, the translation operator $T_x$ acts on a function
$f:\Rdst \to \bC$ by,
\begin{align*}
T_xf(y) := f(y-x). 
\end{align*}
Also, for an invertible matrix
$A \in \Rst^{d \times d}$ we let $\Diltwo{A}$ be the
dilation operator normalized in $L^2(\Rdst)$,
\begin{align}
\Diltwo{A}f(x):= \abs{\det{A}}^{-1/2} f(A^{-1}x).
\end{align}

For $w \in \Rdst$, let $e_w(x) := e^{2 \pi i x w}$, where the multiplication
in $xw$ is the dot product. We use the following normalization of the Fourier
transform
of a function $f: \Rdst \to \bC$,
\begin{align*}
\hat{f}(w) := \int_\Rdst f(x) \overline{e_w(x)} dx.
\end{align*}

Let us now present the construction and main result from \cite{alcamo04-1}.

\begin{itemize}
\item Let $A \in \Rst^{d \times d}$ be an expansive matrix.

\item Select a bounded set $V \subseteq \Rdst$ such that $0 \in V^\circ$
and $\partial V$ has null-measure.

\item Select a function $h \in C^\infty(\Rdst)$ such that
\begin{align*}
\inf_{x \in Q} \abs{h(x)} > 0,
\qquad
Q := (A^tV) \setminus V,
\end{align*}
and
\begin{align*}
0 \notin \supp(h) \subseteq \overline{B}_r,
\end{align*}
where $\overline{B}_r$ is a closed Euclidean ball of radius $0<r<+\infty$
(not necessarily centered at the origin).
\item Select a relatively separated set $\Lambda \subseteq \Rdst$
such that,
\begin{align*}
\rho(\Lambda) < \frac{1}{4 r}.
\end{align*}
\item Let $\psi(x) := \int_\Rdst h(w) e^{2 \pi i w x} dw$
be the inverse Fourier transform of $h$.
\end{itemize}

The possible windows $\psi$ generated by this method will be called
\emph{generalized painless wavelets}. The class of all such windows
will be denoted by $\wavclass$. When we refer to the class $\wavclass$
we always assume that $A$ and $\Lambda$ satisfy the conditions above.

\begin{remark}
\label{rem_class_not_empty}
It is easy to verify that given any expansive matrix $A$ and any well-spread set
$\Lambda$, the class $\wavclass$ is not empty. That is, we can use the
construction above to produce a function $\psi \in \wavclass$. Indeed, since
$\Lambda$ is well-spread, $\rho(\Lambda) < +\infty$, so it suffices to let $V$
be a sufficiently small neighborhood of the origin.
\end{remark}

The following result was proven in \cite{alcamo04-1}. For examples of the
construction see \cite{alcamo04-1}.
\begin{theo}
\label{th:l2}
Let $\psi \in \wavclass$. Then the affine system,
\begin{align}
\label{eq_Psi_pain}
\Psi \supereq
\set{\Diltwo{A^j} T_\lambda \psi}
{j \in \Zst, \lambda \in \Lambda},
\end{align}
is a frame of $L^2(\Rdst)$. That is, it
satisfies the following frame inequality
for some constants $0<A \leq B < +\infty$,
\begin{align*}
A \norm{f}_2^2
\leq
\sum_{j \in \Zst}
\sum_{\lambda \in \Lambda}
\abs{\ip{f}{\Diltwo{A^j} T_\lambda \psi}}^2
\leq B \norm{f}_2^2,
\qquad
(f \in L^2(\Rdst)).
\end{align*}
\end{theo}
\section{Anisotropic Besov-Triebel-Lizorkin spaces}
\label{sec:spaces}
\subsection{Function and sequence spaces}
We now introduce the class of anisotropic Besov-Triebel-Lizorkin spaces and
recall some basic facts about them \cite{lotr79, sctr73, gata02, tr04-2, bo05-1,
boho06, bo07b, bo08-3}). For a comprehensive discussion of the literature on
anisotropic function spaces see \cite{tr04-2}.
We will mainly follow the approach in \cite{bo05-1, boho06, bo07b,
bo08-3} that considers general expansive dilations and generalizes to the
anisotropic setting some of the fundamental results of Frazier and Jawerth
\cite{frja85, frja90} (see also \cite{frjawe91, gihaholawewe02}).
Together with each functional space we introduce a corresponding
sequence space that will measure the size of the coefficients in atomic
decompositions.

Let $A \in \Rst^{d \times d}$ be an expansive matrix.
Let $\varphi \in \Schwartz(\Rdst)$ be such that
\begin{align*}
&\supp(\hat{\varphi}) \subseteq ( [-1/2,1/2]^d \setminus \sett{0} ),
\\
&\sup_{j \in \Zst} \abs{\hat{\varphi}((A^*)^jw)} >0,
\quad \mbox{for all $w \in \Rdst \setminus \sett{0}$}.
\end{align*}
Let $\mu$ be a measure on $\Rdst$ that is doubling with respect to the
seminorm induced by $A$ (see \cite{bo05-1}, in
particular the Lebesgue measure is adequate \cite[Remark 2.1]{bo05-1}).
We are mainly interested in the case of the Lebesgue measure, but the theory is
available for more general measures. The reader interested in this level of
generality should keep in mind that all the estimates throughout the article
depend on the choice of the measure $\mu$. When we want to emphasize
that we let $\mu$ be the Lebesgue measure we say we are in the unweighted case.

For $\alpha \in \Rst$ and $0<p,q \leq \infty$, the homogeneous, weighted,
anisotropic Besov space $\bes = \bes(\Rdst, A, \mu)$ is defined as the
collection of all distributions modulo polynomials $f \in \Schwartz' /
\mathcal{P}$ such that,
\begin{align}
\label{deffbes}
\norm{f}_\bes :=
\left(
\sum_{j \in \Zst}
\abs{\det{A}}^{-j (\alpha+1/2) q} \norm{f*\Diltwo{A^j}\varphi}_{L^p(\mu)}^q
\right)^{1/q}
< +\infty,
\end{align}
with the usual modifications when $q=\infty$. In \cite{bo05-1}
it is proved that $\bes$ is a quasi-Banach space (Banach for $p,q \geq 1$)
and that it is independent of the particular choice of $\varphi$ in the sense
that different choices yield the same space with equivalent norms.

The sequence space $\besseq(\Zdst)=\besseq(\Zdst,A,\mu)$ consists of all
the sequences $a \in {\bC}^{\Zst \times \Zdst}$ such that,
\begin{align}
\label{deffbesse}
\norm{a}_\besseq :=
\left(
\sum_{j\in \Zst}
\abs{\det{A}}^{-j (\alpha+1/2) q}
\Bignorm{
\sum_{k \in \Zdst}
\abs{a_{j,k}} \chi_{(A^j([0,1]^d+k))}}_{L^p(\mu)}^q
\right)^{1/q}
< +\infty,
\end{align}
with the usual modifications when $q=\infty$.

For $\alpha \in \Rst$, $0<q \leq +\infty$ and $0<p<+\infty$,
the anisotropic homogeneous Triebel-Lizorkin space
$\trie=\trie(\Rdst, A, \mu)$ is defined similarly, this time using the norm,
\begin{align}
\label{defftrie}
\norm{f}_\trie :=
\wholenorm{
\left(
\sum_{j \in \Zst} \abs{\det{A}}^{-j (\alpha +1/2) q}
\abs{f*\Diltwo{A^j}\varphi}^q
\right)^{1/q}
}_{L^p(\mu)},
\end{align}
with the usual modifications when $q=+\infty$.
The norm of the associated sequence space on $\Zst \times \Zdst$ is given
by,
\begin{align}
\label{defftrieseq}
\norm{a}_\trieseq :=
\wholenorm{
\left(
\sum_{j\in\Zst}
\abs{\det{A}}^{-j (\alpha+1/2) q}
\sum_{k \in \Zdst}
\abs{a_{j k}}^q \chi_{(A^j([0,1]^d+k))}
\right)^{1/q}
}_{L^p(\mu)}.
\end{align}
The definitions of the spaces $\trie$ and $\trieseq$ for $p=+\infty$ are
more technical. For $\alpha \in \Rst$, $0<q \leq +\infty$,
$\dot{F}^{\alpha,q}_{+\infty}$ is defined using the norm,
\begin{align}
\label{defftrieqinf}
\norm{f}_{\dot{F}^{\alpha,q}_{+\infty}} :=
\sup_{s \in \Zst, k \in \Zdst}
\left(
\frac{1}{\mu(Q_{s,k})}
\int_{Q_{s,k}}
\sum_{j=-\infty}^s \abs{\det{A}}^{-j (\alpha +1/2)q}
\abs{f*\Diltwo{A^j}\varphi(x)}^q d\mu(x)
\right)^{1/q},
\end{align}
where $Q_{j,k} := A^j([0,1]^d+k)$.
When $q=+\infty$ this should be interpreted as,
\begin{align}
\label{defftrieinfinf}
\norm{f}_{\dot{F}^{\alpha,+\infty}_{+\infty}} :=
\sup_{j \in \Zst}
\abs{\det{A}}^{-j (\alpha+1/2)}
\norm{f*\Diltwo{A^j}\varphi}_{L^\infty(\mu)}.
\end{align}
The corresponding space $\dot{f}^{\alpha,q}_{+\infty}(\Zdst)$
of sequences on $\Zst \times \Zdst$ is defined using the norm,
\begin{align}
\label{defftrieqseqinf}
\norm{a}_{\dot{f}^{\alpha,q}_{+\infty}} :=
\sup_{s \in \Zst, k \in \Zdst}
\left(
\frac{1}{\mu(Q_{s,k})}
\int_{Q_{s,k}}
\sum_{j=-\infty}^s \abs{\det{A}}^{-j (\alpha+1/2) q}
\abs{a_{j,k}}^q\chi_{Q_{j,k}}(x)^q d\mu(x)
\right)^{1/q}.
\end{align}
When $q=+\infty$ this should be interpreted as,
\begin{align}
\label{defftrieqseinfinf}
\norm{a}_{\dot{f}^{\alpha,+\infty}_{+\infty}} :=
\sup_{j \in \Zst, k \in \Zdst}
\abs{\det{A}}^{-j (\alpha+1/2)}
\abs{a_{j,k}}.
\end{align}
For a discussion about these definitions see \cite{boho06,bo07b, bo08-3}.
The inhomogeneous variants of all these spaces are defined similarly
(see \cite{bo05-1, boho06, bo07b}), but for simplicity we will only treat
the homogeneous case.

The class of spaces introduced above will be generically called the family
of \emph{anisotropic spaces of Besov-Triebel-Lizorkin type} associated with a
certain dilation $A$ and a measure $\mu$. A member of that family will be
denoted by $\Essp$ while its corresponding sequence space will be denoted by
$\essp$. Each of the spaces $\Essp$ is continuously embedded into $\Schwartz'/
\mathcal{P}$ and is a quasi-Banach space (see \cite{bo05-1, boho06, bo07b}).

\subsection{Sequence spaces on more general index sets}
\label{sec:seq_lambda}
Let $\Lambda \subseteq \Rdst$ be a relatively separated set. 
In order to measure the size of the coefficients associated with the wavelet
system in Section \ref{sec:cons}, we now define a space of sequences indexed
$\Zst \times \Lambda$.

Let $n := \max_{k \in \Zdst} \#(\Lambda \cap ([0,1)^d+\sett{k}))$. Since
each of the cubes $[0,1)^d+\sett{k}$ contains at most $n$ points of
$\Lambda$, it follows that we can split $\Lambda$ into 
a disjoint union of $n$ subsets,
\begin{align*}
\Lambda = \Lambda^1 \cup \ldots \cup \Lambda^n,
\end{align*}
where each subset is parametrized by a set $I^s \subseteq \Zdst$,
\begin{align*}
\Lambda^s = \set{\lambda^s_k}{k \in I^s},
\qquad (1 \leq s \leq n), 
\end{align*}
and satisfies,
\begin{align}
\label{eq_split_control}
\abs{\lambda^s_k - k}_\infty \leq 1,
\qquad (1 \leq s \leq n, k \in I^s).
\end{align}
Given a sequence of complex numbers
$c = \set{c_{j,\lambda}}{j \in \Zst, \lambda \in \Lambda}$,
we define auxiliary sequences $c^1, \ldots c^s$
with indexes on $\Zst \times \Zdst$ by,
\begin{align*}
c^s_{j,k} :=
\left\{
\begin{array}{ll}
c_{j,\lambda_k}, & \mbox{if $k \in I^s$,} \\
0, & \mbox{if $k \notin I^s$.}
\end{array}
\right.
\end{align*}
We then define the space $\essp(\Lambda)$ as the set of all sequences
$c \in {\bC}^{\Zst \times \Lambda}$ such that the corresponding auxiliary
sequences
$c^1, \ldots, c^s$ belong to $\essp(\Zdst)$.
We endow the space $\essp(\Lambda)$ with the norm,
\begin{align*}
\norm{c}_{\essp(\Lambda)}
= \norm{c^1}_{\essp(\Zdst)} + \ldots + \norm{c^s}_{\essp(\Zdst)}.
\end{align*}
When the underlying set $\Lambda$ is clear from the context we will
write $\essp$ instead of $\essp(\Lambda)$.
Note that the definition depends
on a specific decomposition of the set $\Lambda$. This is a bit unsatisfactory
but it will be sufficient for
the purpose of this article since we will keep the set $\Lambda$
fixed. The question of the naturality of the definition is rather involved
and beyond the scope of this article. We refer the reader to
\cite{fegr85, fe87} for some results in this direction
- see also \cite[Lemma 3.5]{fegr89}.
\subsection{Smooth atomic decomposition}
Let $\psi,\varpsi \in \Schwartz$ be such that,
\begin{align*}
\supp(\hat{\psi}),\supp(\hat{\varpsi}) \subseteq [-1/2,1/2]^d \setminus
\sett{0}.
\end{align*}
Assume also that $\psi,\varpsi$ satisfy the \emph{Calder\'on condition},
\begin{align*}
\sum_{j \in \Zst} \hat{\psi}((A^*)^jw)\overline{\hat{\varpsi}((A^*)^jw)} = 1,
\quad \mbox{for all $w \in \Rdst \setminus \sett{0}$}.
\end{align*}
The work of Bownik and Ho \cite{bo05-1, boho06, bo07b, bo08-3}, which extends to the anisotropic case the
fundamental results of Frazier and Jawerth \cite{frja85, frja90}, implies that the windows $\psi, \tau$ provide the following
atomic decomposition of Besov and Triebel-Lizorkin spaces.
\begin{theo}
\label{th_at_desc_bow}
Let $0 < p,q \leq +\infty, \alpha \in \Rst$, let
$\Essp$ be an anisotropic space of Besov-Triebel-Lizorkin type and let
$\essp$ be the corresponding sequence space.
Then the \emph{analysis} and \emph{synthesis} operators,
\begin{align*}
&\coefpsi: \Essp \to \essp(\Zdst),
\qquad
\coefpsi(f) = \left( \ip{f}{\Diltwo{A^j}T_{k} \psi} \right)_{j,k},
\\
&\recvarpsi: \essp(\Zdst) \to \Essp,
\qquad
\recvarpsi(c) = \sum_j \sum_k c_{j,k} \Diltwo{A^ j}T_{k} \varpsi,
\end{align*}
are bounded. Moreover, $\recvarpsi \circ \coefpsi$ is the identity on $\Essp$.

Hence, each $f \in \Essp$ admits the expansion,
\begin{align*}
f = \sum_{j \in \Zst, k \in \Zdst} \ip{f}{\Diltwo{A^ j}T_{k}
\psi}\Diltwo{A^j}T_{k}\varpsi.
\end{align*}
Convergence takes place in the $\Schwartz'/\mathcal{P}$ topology and,
for $p,q < +\infty$, also in the norm of $\Essp$. (See \cite[Lemmas 2.6
and 2.8]{boho06} for a discussion about the precise meaning of the
$\Schwartz'/\mathcal{P}$ convergence.)
\end{theo}
\begin{remark}
As a consequence of the atomic decomposition, the following norm equivalence
holds,
\begin{align*}
\norm{f}_\Essp \approx \bignorm{(\ip{f}{\Diltwo{A^ j}T_{k}\psi})_{j,k}}_\essp,
\qquad (f \in \Essp).
\end{align*}
\end{remark}

\subsection{Molecules and general decompositions}
\label{sec:mol}
\begin{deff}
\label{def_molecules}
Given $L>0$, integers $M,N>0$ and a relatively separated set $\Lambda
\subseteq \Rdst$,
a family of functions,
\begin{align*}
\Psi \supereq
\set{\psi_{j,\lambda}: \Rdst \to \bC}
{j \in \Zst, \lambda \in \Lambda}, 
\end{align*}
is called a set of $(L,M,N)$-molecules if
each function has continuous derivatives up to order $M$ and satisfies,
\begin{align*}
&\abs{\partial^\beta(\Diltwo{A^{-j}} \psi_{j,\lambda})(x)} \leq
(1+\abs{x-\lambda})^{-L},
\quad
\mbox{for all $\abs{\beta} \leq M$},
\\
&\int_\Rdst x^\beta \psi_{j,\lambda}(x) dx = 0,
\quad
\mbox{for all $\abs{\beta} \leq N$}.
\end{align*}
\end{deff}
\begin{remark}
We stress that the functions in the definition of family of molecules do not
need to be dilated and translated versions of a single function. Throughout
the article, it should not be assumed, unless explicitly stated, that a function
denoted by $\psi_{j,\lambda}$ is equal to $\Diltwo{A^j}T_\lambda \psi$.
\end{remark}

The definition presented here consists of a simplified version of the
anisotropic molecules introduced in \cite{bo05-1,boho06,bo07b}. This will be
sufficient for our purpose. The notion of set of molecules generalizes the one
of set of atoms in the sense that the analysis and synthesis maps
(see Theorem \ref{th_at_desc_bow}) can be extended to these families.
The most important technical point is the justification of the meaningfulness of
the
quantity $\ip{f}{\psi_{j,\lambda}}$ for a molecule
$\psi_{j,\lambda}$ and $f \in \Essp$, since $\psi_{j,\lambda}$ may neither be
in the Schwartz class nor have all its moments vanishing. This point is
thoroughly discussed in \cite{bo05-1,boho06,bo07b}; we refer the reader to these
articles.

The following theorem, which is a minor modification of the corresponding results in the above mentioned articles,
establishes the boundedness of analysis and synthesis for a general molecular system, which is not necessarily a wavelet
system and is indexed by a general relatively separated set $\Lambda$.

\begin{theo}
\label{th_bessel_bounds}
Let $\Essp$ be an anisotropic space of Besov-Triebel-Lizorkin type
$(0 < p,q \leq +\infty, \alpha \in \Rst)$,
let $\Lambda \subseteq \Rdst$ be a relatively separated set and let
$\essp(\Lambda)$ be the corresponding sequence space. Then, there
exist
constants $L, M, N$ and $C$ that only depend on $A$, $p, q,
\alpha$ and $\Lambda$ such that for every set of $(L,M,N)$-molecules,
\begin{align*}
\Psi \supereq
\set{\psi_{j,\lambda}: \Rdst \to \bC}
{j \in \Zst, \lambda \in \Lambda}, 
\end{align*}
the following estimates hold.
\begin{align*}
&\Bignorm{\sum_{j,\lambda} c_{j,\lambda} \psi_{j,\lambda}}_\Essp \leq C
\norm{c}_{\essp(\Lambda)},
\quad
\mbox{for all $c \in \essp(\Lambda)$},
\\
&\bignorm{(\ip{f}{\psi_{j,\lambda}})_{j,\lambda}}_\essp \leq C \norm{f}_\Essp,
\quad
\mbox{for all $f \in \Essp$}.
\end{align*}
Hence, the analysis and synthesis maps,
\begin{align}
\label{eq_operators}
&\coefPsi: \Essp \to \essp(\Lambda),
\qquad
\coefPsi(f) = \left( \ip{f}{\psijl} \right)_{j,\lambda},
\\
&\recPsi: \essp(\Lambda) \to \Essp,
\qquad
\recPsi(c) = \sum_j \sum_\lambda c_{j,\lambda} \psijl,
\end{align}
are bounded.
\end{theo}
\begin{remark}
\label{rem_conv_rec}
The series defining $\recPsi(c)$ converge in the
$\Schwartz'/\mathcal{P}$ topology and,
for $p,q < +\infty$, also in the norm of $\Essp$.
\end{remark}
\begin{proof}[Proof of Theorem \ref{th_bessel_bounds}]
Let us consider first the case $\Lambda=\Zdst$. In this case the result follows
from the results in Section 5 of \cite{bo05-1, boho06, bo07b}. The definition of
molecule given there requires the decay conditions,
\begin{align}
\label{eq_decay_quasi}
&\abs{\partial^\beta(\Diltwo{A^{-j}} \psi_{j,\lambda})(x)} \leq
(1+\quasiA(x-\lambda))^{-L'},
\quad
\mbox{for all $\abs{\beta} \leq M$},
\end{align}
with respect to a certain quasi-norm $\quasiA$ associated with $A$
and a certain constant $L'$. This quasi-norm $\quasiA$ satisfies,
\begin{align*}
\quasiA(x) \lesssim \abs{x}^t, 
\end{align*}
for $\quasiA(x) \geq 1$ and some number $t>0$ that depends on $A$
(see \cite[Lemma 3.2]{bo03-1} or \cite[Lemma 2.2]{boho06}). It follows
that
\begin{align*}
(1+\abs{x-\lambda})^{-tL'}
\lesssim 
(1+\quasiA(x-\lambda))^{-L'}. 
\end{align*}
Hence, for any $L' \geq 0$, the decay condition prescribed by Equation
\eqref{eq_decay_quasi} is satisfied if we take $L \geq t L'$.

We now show how to reduce the case of a general set $\Lambda$ to the one of
$\Lambda=\Zdst$. With the notation from Section \ref{sec:seq_lambda},
Equation \eqref{eq_split_control} implies that for all $1 \leq s \leq n$,
\begin{align*}
\abs{\partial^\beta(\Diltwo{A^{-j}} \psi_{j,\lambda^s_k})(x)}
&\leq (1+\abs{x-\lambda^s_k})^{-L}
\\
&\lesssim 2^L (1+\abs{x-k})^{-L}.
\end{align*}
Hence, if we define the families
$\Psi^s \supereq \set{\psi^s_{j,k}}{j \in \Zst, k \in \Zdst}$ by
$\psi^s_{j,k} = \psi_{j,\lambda^s_k}$ if $k \in I^s$ and 0 otherwise,
it follows that each set $\Psi^s$ is a constant multiple of a family of
molecules. Therefore, the corresponding analysis and
synthesis maps are bounded. Hence,
\begin{align*}
\Bignorm{\sum_{j \in \Zst, \lambda \in \Lambda} c_{j,\lambda}
\psi_{j,\lambda}}_\Essp
&\lesssim
\sum_{s=1}^n
\Bignorm{\sum_{j \in \Zst, k \in \Zdst} c^s_{j,k}
\psi^s_{j,k}}_\Essp
\\
&\lesssim
\sum_{s=1}^n \bignorm{c^s}_{\essp(\Zdst)} = \bignorm{c}_{\essp(\Lambda)}.
\end{align*}
Similarly,
\begin{align*}
\bignorm{(\ip{f}{\psijl})_{j \in \Zst, \lambda \in \Lambda}}_{\essp(\Lambda)}
&=
\sum_{s=1}^n
\bignorm{\left(\ip{f}{\psi^s_{j, k}}\right)_
{j \in \Zst, k \in \Zdst}}_{\essp(\Zdst)}
\\
&\lesssim \norm{f}_\Essp.
\end{align*}
\end{proof}

\section{Balayage of Dirac measures}
\label{sec:bal}
The balayage problem for the ball, consists in deciding if the restriction to the Euclidean ball of the Fourier transform of
an arbitrary measure can be represented as a linear combination of imaginary exponentials along a given, quite arbitrary,
discrete set. In \cite{be66-3}, Beurling proved that this is indeed possible if the gap of the discrete set is small
enough (cf. Eq. \eqref{eq_gap}). In that same article, he proved that if the measure to be represented is very concentrated,
then the coefficients in the corresponding non-harmonic trigonometric expansion of the restriction of its Fourier transform
have fast decay. This is essential to the present article.

We now quote a simplified version of one of the results in \cite{be66-3}.
\begin{theo}
\label{theo_beurling}
Let $r>0$ and let $\Lambda \subseteq \Rdst$ be such that
$\rho(\Lambda) r < 1/4$. Then, there exist constants $C>0$, $0<\consexp<1$,
depending only on $d, \Lambda$ and $r$ such that for every $w \in \Rdst$, there
exists a sequence $\set{a_\lambda(w)}{\lambda \in \Lambda} \subseteq \bC$ such
that
\begin{align*}
&e_w(x) = \sum_{\lambda \in \Lambda} a_\lambda(w) e_\lambda(x),
\quad
\mbox{for all $x \in \overline{B}_r(0)$},
\\
&\abs{a_\lambda(w)} \leq C e^{-\consexp \abs{w-\lambda}^{1/2}}.
\end{align*}
\end{theo}
Theorem \ref{theo_beurling} follows by applying the estimates obtained in
Equation (24) and (25) from \cite{be66-3} to the function
$\Omega(w):=\consexp \abs{w}^{1/2}$ and a suitably small constant $c>0$
(cf. Equation (15) in \cite{be66-3}). The result in \cite{be66-3}
applies to $r=1$. Theorem \ref{theo_beurling} follows after rescaling.

For convenience, we give the following straightforward variation of Theorem
\ref{theo_beurling}.
\begin{coro}
\label{coro_beurling}
Let $r>0$, let $\Lambda \subseteq \Rdst$ be such that
$\rho(\Lambda) r < 1/4$ and $x_0 \in \Rdst$. Then, there exist constants $C>0$,
$0<\consexp<1$, depending only on $d, \Lambda$ and $r$ such that for every $w
\in \Rdst$, there exists a sequence $\set{a_\lambda(w)}{\lambda \in \Lambda}
\subseteq \bC$ such that
\begin{align*}
&e_{-w}(x) = \sum_{\lambda \in \Lambda} a_\lambda(w) e_{-\lambda}(x),
\quad
\mbox{for all $x \in \overline{B}_r(x_0)$},
\\
&\abs{a_\lambda(w)} \leq C e^{-\consexp \abs{w-\lambda}^{1/2}}.
\end{align*}
\end{coro}
\begin{proof}
This follows by translating and conjugating the equality in Theorem
\ref{theo_beurling}. That does not affect the absolute value of
the coefficients $a_\lambda(w)$ and hence the decay condition is preserved.
\end{proof}

\section{The generalized painless method in Besov-Triebel-Lizorkin spaces}
\label{sec:gen}
We can now show that the construction from Section \ref{sec:cons} yields
an atomic decomposition for the whole class of anisotropic
Besov-Triebel-Lizorkin spaces.

\begin{theo}
\label{th_bandlimited}
Let $\psi \in \wavclass$ (cf. Section \ref{sec:cons}) and let,
\begin{align*}
\Psi \supereq
\set{\psijl := \Diltwo{A^j}T_\lambda \psi}
{j \in \Zst, \lambda \in \Lambda}.
\end{align*}
Then, there exists a family of band-limited functions,
\begin{align*}
\tilde{\Psi} \supereq
\set{\tilpsijl}{j \in \Zst,\lambda \in \Lambda},
\end{align*}
such that the following statements hold for each anisotropic space of
Besov-Triebel-Lizorkin type
$\Essp$, with $0 < p,q \leq \infty$ and $\alpha \in \Rst$.
\begin{itemize}
\item[(a)] The following analysis (coefficient) and synthesis (reconstruction)
operators are bounded.
\begin{align*}
&\coefPsi: \Essp \to \essp(\Lambda), &\qquad
&f \mapsto \left( \ip{f}{\psi_{j,\lambda}} \right)_{j \in \Zst, \lambda \in
\Lambda},
\\
&\coeftilPsi: \Essp \to \essp(\Lambda), &\qquad
&f \mapsto \left( \ip{f}{\tilpsijl} \right)_{j \in \Zst, \lambda \in
\Lambda},
\\
&\recPsi: \essp(\Lambda) \to \Essp, &\qquad
&c \mapsto \sum_{j,\lambda} c_{j,\lambda} \psi_{j,\lambda},
\\
&\rectilPsi: \essp(\Lambda) \to \Essp, &\qquad
&c \mapsto \sum_{j,\lambda} c_{j,\lambda} \tilpsijl.
\end{align*}
\item[(b)] Every $f \in \Essp$ admits the expansions,
\begin{align}
\label{eq_exp_irreg}
f &= \sum_{j \in \Zst} \sum_{\lambda \in \Lambda} 
\ip{f}{\Diltwo{A^j}T_\lambda \psi} \tilpsijl
\\
\nonumber
&= \sum_{j \in \Zst} \sum_{\lambda \in \Lambda} 
\ip{f}{\tilpsijl} \Diltwo{A^j}T_\lambda \psi.
\end{align}
Convergence takes place in the $\Schwartz'/\mathcal{P}$ topology and,
for $p,q < +\infty$, also in the norm of $\Essp$.
\item[(c)] The following norm equivalence holds,
\begin{align*}
\norm{f}_\Essp \approx
\Bignorm{\left( \ip{f}{\psi_{j,\lambda}}
\right)_{j,\lambda}}_\essp
\approx
\Bignorm{\left( \ip{f}{\tilpsijl}
\right)_{j,\lambda}}_\essp
,
\quad (f \in \Essp).
\end{align*}
\end{itemize}
\end{theo}
\begin{remark}
For all $(L,M,N)$, the family $\{\tilpsijl\}_{j, \lambda}$ is a multiple of a
set of molecules in the sense of Definition \ref{def_molecules}. Hence, the
operators in (a) are well-defined. In addition, the dual functions have the
form,
\begin{align*}
\tilpsijl = \Diltwo{A^j}{\tilpsil}, 
\end{align*}
for a certain family of functions $\set{\tilpsil}{\lambda \in \Lambda}$.
\end{remark}

\begin{proof}[Proof of Theorem \ref{th_bandlimited}]
Let us adopt the notation from Section \ref{sec:cons}.
It follows from the construction of $\psi$ that
$\sum_{j \in \Zst} \abs{\hat{\psi}((A^*)^jw)}^2 \approx 1$, for $w \not=0$,
and that $\supp(\hat{\psi}) \subseteq \overline{B}_r(w_0)$, for some
$w_0 \in \Rdst$. In addition, since $0 \notin \supp(\hat{\psi})$,
we have that $\supp(\hat{\psi}) \subseteq [-b/2,b/2]^d \setminus \sett{0}$, for some
$b>0$.

We now produce a ``dual window'' $\varpsi$ so that $\psi, \varpsi$ satisfy the Calder\'on condition.
This follows the now standard method of Frazier and Jawerth. Let $\varpsi$ be such that,
\begin{align*}
\hat{\varpsi}(w) := b^d \hat{\psi}(w) \left(\sum_{j \in \Zst}
\abs{\hat{\psi}((A^*)^jw)}^2 \right)^{-1}, 
\qquad (w \not=0),
\end{align*}
so that,
\begin{align*}
&\supp(\hat{\varpsi}) \subseteq B_r(w_0) \cap [-b/2,b/2]^d \setminus \sett{0},
\\
&\sum_{j \in \Zst}
\hat{\psi}((A^*)^jw)
\overline{\hat{\varpsi}((A^*)^jw)} = b^{d},
\qquad (w \not=0).
\end{align*}
It is easy to see that $\varpsi \in \mathcal{S}(\Rdst)$
(see for example \cite[Lemma 3.6]{bo05-1} or \cite{boho06}).
By rescaling the result in Theorem \ref{th_at_desc_bow}, we see that
the windows $\psi, \varpsi$ provide the following lattice-based expansion for every
anisotropic Besov-Triebel-Lizorkin space $\Essp$,
\begin{align}
\label{eq_exp}
f &= \sum_{j \in \Zst, k \in \Zdst}
\ip{f}{\Diltwo{A^ j}T_{\frac{k}{b}} \psi}
\Diltwo{A^j}T_{\frac{k}{b}}\varpsi
\\
&=
\nonumber
\sum_{j \in \Zst, k \in \Zdst}
\ip{f}{\Diltwo{A^ j}T_{\frac{k}{b}} \varpsi}
\Diltwo{A^j}T_{\frac{k}{b}}\psi,
\qquad (f \in \Essp),
\end{align}
with convergence in the $\mathcal{S'}/\mathcal{P}$ topology.
For $j \in \Zst$ and $\lambda \in \Lambda$,
let $\psi_{j,\lambda} := \Diltwo{A^j}T_{\lambda} \psi$.
For each $k \in \Zdst$, Corollary \ref{coro_beurling} gives a sequence
$\set{a_{\lambda,k}}{\lambda \in \Lambda} \subseteq \bC$ such that,
\begin{align}
\nonumber
&e_{-k/b}(w) = \sum_{\lambda \in \Lambda} a_{\lambda,k} e_{-\lambda}(w),
\quad
\mbox{for all $w \in B_r(w_0)$},
\\
\label{eq_coef_alk}
&\abs{a_{\lambda,k}} \lesssim e^{-\consexp \abs{\lambda-k/b}^{1/2}} .
\end{align}
Since $\hat{\psi}$ and $\hat{\varpsi}$ are both supported on
$\overline{B}_r(w_0)$, it follows that,
\begin{align}
\label{eq_tkrl}
T_{\frac{k}{b}}\psi = \sum_{\lambda \in \Lambda} a_{\lambda,k} T_\lambda \psi. 
\end{align}
For $j \in \Zst$ and $\lambda \in \Lambda$, let,
\begin{align*}
\tilde{\psi}_{j,\lambda}:=
\sum_{k\in\Zdst} a_{\lambda,k} \Diltwo{A^j}T_{\frac{k}{b}}\varpsi
=
\Diltwo{A^j} \tilpsil,
\end{align*}
where,
\begin{align}
\label{eq_tiltkrl}
\tilpsil := \sum_{k\in\Zdst} a_{\lambda,k} T_{\frac{k}{b}}\varpsi.
\end{align}
Formally replacing Equation \eqref{eq_tkrl} in Equation \eqref{eq_exp}
yields the expansions in Equation \eqref{eq_exp_irreg}. For example,
the first expansion in Equation \eqref{eq_exp_irreg} follows (formally)
from the corresponding expansion in Equation \eqref{eq_exp}
by the following computation.
\begin{align}
\nonumber
& \sum_{k \in \Zdst}
\ip{f}{\Diltwo{A^j}{T_{\frac{k}{b}} \psi}}
\Diltwo{A^j}T_{\frac{k}{b}}\varpsi
\\
\nonumber
&\qquad\qquad=
\sum_{k \in \Zdst}
\ip{f}{\Diltwo{A^j}{\sum_{\lambda \in \Lambda} a_{\lambda,k} T_\lambda \psi}}
\Diltwo{A^j}T_{\frac{k}{b}}\varpsi
\\
%\label{eq_changekl1}
\nonumber
& \qquad\qquad =
\sum_{k \in \Zdst}
\sum_{\lambda \in \Lambda} a_{\lambda,k}
\ip{f}{\Diltwo{A^j}{ T_\lambda \psi}}
\Diltwo{A^j}T_{\frac{k}{b}}\varpsi
\\
\label{eq_changekl2}
& \qquad\qquad =
\sum_{\lambda \in \Lambda}
\sum_{k \in \Zdst}
a_{\lambda,k}
\ip{f}{\Diltwo{A^ j}{ T_\lambda \psi}}
\Diltwo{A^j}T_{\frac{k}{b}}\varpsi
\\
\nonumber
& \qquad\qquad =
\sum_{\lambda \in \Lambda}
\ip{f}{\Diltwo{A^j}{ T_\lambda \psi}}
\sum_{k \in \Zdst} a_{\lambda,k} \Diltwo{A^j}T_{\frac{k}{b}}\varpsi
\\
\nonumber
& \qquad\qquad =
\sum_{\lambda \in \Lambda}
\ip{f}{\Diltwo{A^j}{ T_\lambda \psi}}
\Diltwo{A^j} \sum_{k \in \Zdst} a_{\lambda,k} T_{\frac{k}{b}}\varpsi
\\
\nonumber
& \qquad\qquad =
\sum_{\lambda \in \Lambda}
\ip{f}{\Diltwo{A^j}{ T_\lambda \psi}} \tilpsijl.
\end{align}
Let us now observe that these formal operations are indeed valid. We discuss the
validity of the first expansion in Equation \eqref{eq_exp_irreg}, the second one
being analogous. Because of the fast decay of the numbers in Equation
\eqref{eq_coef_alk}, the series in Equations \eqref{eq_tkrl} and
\eqref{eq_tiltkrl} converge absolutely in $L^2$. Hence, the dilation
operator $\Diltwo{A^j}$ can be interchanged with both summations. We now discuss
the interchange of summation in $k$ and $\lambda$ in Equation
\eqref{eq_changekl2}. For $\lambda \in \Lambda$,
\begin{align*}
\ip{f}{\Diltwo{A^j}{T_\lambda \psi}}
=
\ip{f}{T_{A^j \lambda} \Diltwo{A^j} \psi}
=
\ip{\hat{f}\cdot\overline{\widehat{\Diltwo{A^j} \psi}}}{e_{A^j\lambda}}.
\end{align*}
Note that this computation is valid since $\widehat{\Diltwo{A^j} \psi}$
vanishes in a neighborhood of the origin. Since
$\hat{f} \cdot \overline{\widehat{\Diltwo{A^j} \psi}}$ is a compactly supported
distribution, its distributional Fourier transform is a smooth function with at
most polynomial growth. This shows that,
\begin{align*}
\abs{\ip{f}{\Diltwo{A^j}{T_\lambda \psi}}}
\lesssim
(1+\abs{A^j \lambda})^s
\lesssim
(1+\abs{\lambda})^s,
\end{align*}
for some constant $s>0$ (where the implicit constants of course
depend on $j$). This growth estimate together with the fast decay
of the coefficients $a_{\lambda,k}$ in Equation \eqref{eq_coef_alk}
justifies the change of the summation order.

We now prove that the analysis and synthesis operators in item (a) are
bounded. Since $\psi \in \Schwartz$ and $\hat{\psi}$ vanishes near the origin,
it is clear that for each $(L,M,N)$ the set $\set{\psi_{j,\lambda}}{j \in
\Zst,\lambda \in \Lambda}$ is a constant multiple of a set of
$(L,M,N)$-molecules. 
Hence, by Theorem \ref{th_bessel_bounds},
the operators $\coefPsi, \recPsi$ are bounded on each of the
relevant spaces.

The functions $\set{\tilpsil}{\lambda \in \Lambda}$ have all their moments
vanishing. In addition, since $\varpsi \in \Schwartz$, for every $L >0$ and
every multi-index $\beta \in \Nst_0^d$,
\begin{align*}
\abs{\partial^\beta(\varpsi)(x)}
\lesssim \w{x}{-L}.
\end{align*}
Using the estimate in Equation \eqref{eq_coef_alk},
we see that the functions $\set{\tilpsil}{\lambda \in \Lambda}$ satisfy,
\begin{align*}
\abs{\partial^\beta(\tilpsil)(x)}
&\leq \sum_{k\in\Zdst} \abs{a_{\lambda,k}}
\abs{\partial^\beta(\varpsi)(x-k/b)}
\\
&\lesssim \sum_{k\in\Zdst} e^{-\consexp\abs{k/b-\lambda}^{1/2}} \w{x-k/b}{-L}
\\
&\lesssim \w{x-\lambda}{-L}
\sum_{k\in\Zdst} e^{-\consexp\abs{k/b-\lambda}^{1/2}} \w{\lambda-k/b}{L}
\\
&\lesssim \w{x-\lambda}{-L}.
\end{align*}
From that estimate it follows that for each $(L,M,N)$, the family
\begin{align*}
\tilPsi \supereq
\set{\tilpsijl = \Diltwo{A^j} \tilpsil}
{j \in \Zst, \lambda \in \Lambda}
\end{align*}
is also a multiple of a set of molecules
(cf. Definition \ref{def_molecules}).
Therefore, Theorem \ref{th_bessel_bounds} implies that
the operators $\coeftilPsi, \rectilPsi$ are bounded on each of the
relevant spaces. Finally, the claimed norm equivalence follows from the fact
that $\coefPsi$ and $\coeftilPsi$ have bounded right-inverses (namely,
$\rectilPsi$ and $\recPsi$).
\end{proof}
\begin{coro}
\label{coro_bandlimited}
Let $\Lambda \subseteq \Rdst$ be any well-spread set and let
$A \in \Rst^{d\times d}$ be an expansive matrix. Then there exists 
a Schwartz class function $\psi$ with Fourier transform supported
on a compact set not containing the origin, such that the irregular wavelet
system,
\begin{align*}
\set{\Diltwo{A^j}T_{\lambda} \psi}{j\in\Zst, \lambda \in \Lambda},
\end{align*}
provides an atomic decomposition for the whole class of anisotropic Besov and
Triebel-Lizorkin spaces $\Essp, 0 < p,q \leq \infty, \alpha \in \Rst$
(in the precise sense of Theorem \ref{th_bandlimited}).
\end{coro}
\begin{proof}
This follows immediately from Theorem \ref{th_bandlimited} and Remark
\ref{rem_class_not_empty}.
\end{proof}

\section{Compactly supported non-uniform wavelets}
\label{sec:compact}
In this section we will combine Theorem \ref{th_bandlimited} with a
perturbation argument to show that any expansive matrix and any a well-spread
set of translation nodes admit a compactly supported wavelet frame. The
perturbation method is quite standard (see for example \cite{goza97,pe00,
aibego01,krpe02, nira11}). However, we are interested in obtaining results 
that hold uniformly for a whole range of Triebel-Lizorkin spaces. This requires
to deal with a number of technical matters. To this end, we first introduce the
relevant tools.
\subsection{Some technical tools for Triebel-Lizorkin spaces}
\label{sec:tech}
In order to carry out the construction in this section we will need a number of
technical tools concerning duality and interpolation in
anisotropic Besov-Triebel-Lizorkin spaces. In the Triebel-Lizorkin
case these have been developed in \cite{bo08-3}. The Besov case is technically
much simpler, but it does not seem to have been explicitly treated in the
literature. Because of this, from now on we restrict ourselves to the
Triebel-Lizorkin case. Moreover, to avoid certain technicalities
with the pairing $\ip{\cdot}{\cdot}$ we further restrict ourselves
to the unweighted case $\mu=dx$.

The wavelet system that we will construct in Section \ref{sec:compact} may not
be a set of molecules (see Definition \ref{def_molecules}). Hence, the meaning
of the coefficient mapping needs to be clarified. In \cite{bo08-3}, Bownik has
extended the pairing $\ip{\cdot}{\cdot}$ to $\Essp \times \Esspdual$,
$(1 \leq p,q < +\infty)$ and has moreover characterized the dual space of
$\Essp$ by means of it. Using this extension, the analysis map $f \mapsto
\coefPhi(f) := (\ip{f}{\phijl})_{j\in \Zst, \lambda \in \Lambda}$ is
well-defined on $\Essp$ if $\phijl \in \Esspdual$ for all $j,\lambda$.
(Here, $p'$ denotes de H\"older conjugate of $p$, given by $1/p+1/p'=1$.)

The following result is Theorem 6.2 in \cite{bo08-3}.
\begin{theo}[Bownik]
\label{th_interp}
Let $\Esspz, \Esspo$ be (anisotropic, homogeneous) Triebel-Lizorkin spaces and
let $\esspz, \esspo$ be the corresponding sequence spaces on $\Zst \times
\Zdst$, with
$\alpha_0, \alpha_1 \in \Rst$, $0<p_0,q_0<+\infty$,
$0<p_1,q_1 \leq +\infty$. Then for $0<\theta<1$ we can identify the complex
interpolation spaces,
\begin{align*}
[\esspz, \esspo]_\theta = \essp,
\\
[\Esspz, \Esspo]_\theta = \Essp,
\end{align*}
where $1/p = (1-\theta) / p_0 + \theta / p_1$,
$1/q = (1-\theta) / q_0 + \theta /q_1$,
and $\alpha = (1-\theta) \alpha_0 + \theta \alpha_1$.
\end{theo}
\begin{remark}
The statement $[\Esspz, \Esspo]_\theta = \Essp$ means that
the underlying sets are equal and the norms are equivalent.
The constants in that norm equivalence may depend on
$\alpha, p$ and $q$.
\end{remark}

\begin{coro}
\label{coro_int}
If a linear operator $T$ is bounded on the anisotropic Triebel-Lizorkin spaces
$\Essp$ for all combinations of indexes $\alpha=\pm \alphanot$, $p=1,\infty$,
$q=1,\infty$, then $T$ is bounded on $\Essp$ for the whole
range $-\alphanot \leq \alpha \leq \alphanot$, $1 \leq p,q<\infty$. 
\end{coro}
\begin{remark}
The reason why we exclude the cases $p=\infty$ or $q=\infty$ is that 
Theorem \ref{th_interp} requires $p_0, q_0 < \infty$.
\end{remark}
\begin{proof}[Proof of Corollary \ref{coro_int}]
Let $B, I$ be the sets,
\begin{align*}
B &:= \set{(p,q,\alpha)}{1 \leq p,q \leq \infty,
-\alphanot \leq \alpha \leq \alphanot, \mbox{ and $T$ is bounded on }\Essp},
\\
I &:= \set{(1/p,1/q,\alpha)}{1 \leq p,q < \infty, (p,q,\alpha) \in B}
\subseteq (0,1] \times (0,1] \times [-\alphanot,\alphanot].
\end{align*}
The conclusion will follow if we prove that
$I$ is actually $(0,1] \times (0,1] \times [-\alphanot,\alphanot]$.
Since, by Theorem \ref{th_interp}, $I$ is a convex set,
it will be sufficient to show that $I$ contains certain points.

Since $(1,1,\pm \alphanot), (+\infty,+\infty,\pm \alphanot) \in I$,
using Theorem \ref{th_interp} with $p_0=q_0=1, \alpha_0=\pm\alphanot$
and $p_1=q_1=+\infty, \alpha_1=\pm\alphanot$ it follows that,
\begin{align}
\label{line1}
\set{(1/p,1/p,\pm\alphanot)}{1 \leq p <+\infty} 
=
\set{(x,x,\pm\alphanot)}{0 < x \leq 1}
\subseteq I.
\end{align}

Secondly, since
$(1,1,\pm \alphanot), (+\infty, 1,\pm \alphanot) \in B$,
using Theorem \ref{th_interp} with $p_0=q_0=1, \alpha_0=\pm\alphanot$
and $p_1=+\infty, q_1=1, \alpha_1=\pm\alphanot$ it follows that,
\begin{align}
\label{line2}
\set{(1/p,1/1,\pm\alphanot)}{1 \leq p <+\infty} 
=
\set{(x,1,\pm\alphanot)}{0 < x \leq 1} 
\subseteq I.
\end{align}

Similarly, since
$(1,1,\pm \alphanot), (1,+\infty,\pm \alphanot) \in I$,
using Theorem \ref{th_interp} with $p_0=q_0=1, \alpha_0=\pm\alphanot$
and $p_1=1, q_1=+\infty, \alpha_1=\pm\alphanot$ it follows that,
\begin{align}
\label{line3}
\set{(1/1,1/q,\pm\alphanot)}{1 \leq q <+\infty}
=
\set{(1,y,\pm\alphanot)}{0 < y \leq 1}
\subseteq I.
\end{align}
Since $I$ is a convex set containing the lines in Equations
\eqref{line1}, \eqref{line2}, \eqref{line3}, it follows that,
\begin{align*}
I = (0,1] \times (0,1] \times [-\alphanot,\alphanot],
\end{align*}
as desired.
\end{proof}
\subsection{Construction of compactly supported windows}
By Theorem \ref{th_bandlimited}, the wavelet frame constructed in
Section \ref{sec:cons} provides an atomic decomposition for all the
unweighted anisotropic Triebel-Lizorkin spaces. We will now use this together
with a perturbation argument to show that these spaces also admit an
atomic decomposition generated by a compactly supported window and the same set
of translations and dilations. 
\begin{theo}
\label{th_compact}
Let $\Lambda \subseteq \Rdst$ be any well-spread set and let
$A \in \Rst^{d \times d}$ be an expansive
matrix. Let $N \in \Nst$ and $\alpha_0 >0$. Then there exists a
compactly supported function $\varphi \in C^\infty(\Rdst)$ such that
\begin{align*}
&\int_{\Rdst} x^\beta \varphi(x) dx = 0,
\quad
\mbox{for all $\abs{\beta} \leq N$,}
\end{align*}
and a family of functions,
\begin{align*}
\tilde{\Phi} \supereq \set {\tilphijl}
{j \in \Zst, \lambda  \in \Lambda},
\end{align*}
such that the following statements hold for the
class of unweighted anisotropic Triebel-Lizorkin spaces.
\begin{itemize}
\item[(a)] $\tilde{\Phi} \subseteq \Essp$,
for all $-\alpha_0 \leq \alpha \leq \alpha_0$  and $1 \leq p,q \leq +\infty$.
\item[(b)] Every $f \in \Essp$,
$(-\alpha_0 \leq \alpha \leq \alpha_0, 1 \leq p,q < \infty)$,
admits the norm convergent expansion,
\begin{align*}
f &= \sum_{j \in \Zst, \lambda \in \Lambda}
\ip{f}{\Diltwo{A^j}T_\lambda \varphi} \tilphijl
\\
&= \sum_{j \in \Zst, \lambda \in \Lambda}
\ip{f}{\tilphijl} \Diltwo{A^j}T_\lambda \varphi.
\end{align*}
\item[(c)] Further, for $f \in \Essp$,
$(-\alpha_0 \leq \alpha \leq \alpha_0, 1 \leq p,q < \infty)$,
\begin{align*}
\norm{f}_\Essp \approx
\Bignorm{\left(\ip{f}{\Diltwo{A^j}T_\lambda \varphi}\right)_{j,\lambda}}_\essp
\approx
\Bignorm{\left(\ip{f}{\tilphijl}\right)_{j,\lambda}}_\essp.
\end{align*}
\end{itemize}
\end{theo}
\begin{remark}
The expression $\ip{f}{\tilphijl}$ in (b) is well-defined
because, by (a), $\tilphijl \in \Esspdual$ (cf. Section \ref{sec:tech}).
\end{remark}
Before proving Theorem \ref{th_compact} we need the following lemma.
\begin{lemma}
\label{lemma_per}
Let $\psi \in \wavclass$ and $L,M,N >0$ be given.
Then, for each $\varepsilon>0$ there exists a $C^\infty$ compactly-supported
function $\varphi$, with vanishing moments up to order $N$, such that the
irregular wavelet system,
\begin{align*}
\left\{ \left(\frac{1}{\varepsilon}\right) \Diltwo{A^j}T_{\lambda} (\psi-\varphi) \; \big| \; j\in\Zst, \lambda
\in \Lambda \right\},
\end{align*}
is a set of $(L,M,N)$ molecules.
Consequently, the set $\sett{\Diltwo{A^j}T_{\lambda}
(\psi-\varphi)}_{j,\lambda}$ is an $\varepsilon$-multiple of a set of molecules.
\end{lemma}
\begin{proof}
Let $\sett{\tau_\beta: \abs{\beta} \leq N}$ be a collection of smooth,
compactly supported functions
such that $\int x^\gamma \tau_\beta(x) dx = 1$, if $\gamma=\beta$ and 0
otherwise.
For each $R>0$, let $\eta_R: \Rdst \to [0,1]$ be a $C^\infty$, compactly
supported function such that $\eta_R \equiv 1$ on the ball of radius $R$
and such that $\norm{\partial^\beta(\eta_R)}_\infty \lesssim 1$, for all
$\abs{\beta} \leq M$ (with constants independent of $R$, but possibly
depending on $M$).

Let us fix $R>0$ and define
\begin{align*}
\varphi = \varphi_R := \psi \eta_R - \sum_{\abs{\beta} \leq N} c^R_\beta
\tau_\beta, 
\end{align*}
where $c^R_\beta := \int_{\Rdst} x^\beta \psi(x) \eta_R(x) dx$.
By construction, $\varphi$ is a $C^\infty$, compactly supported function with
vanishing
moments up to order $N$. (The
technique of adjusting the moments of $\psi$ by means
of the functions $\tau_\beta$ is also used in
\cite[Lemma 5.4]{boho06}.)
We will show that if $R$ is large enough, then
$\varphi=\varphi_R$
satisfies the required conditions. To prove this, it suffices to show that
there exists $R>0$
such that
\begin{equation}
\label{small_constants}
\abs{\partial^\beta(\psi-\varphi_R)(x)} \leq \varepsilon (1+\abs{x})^{-L},
\quad \mbox{for all $\abs{\beta} \leq M$ and $x \in \Rdst$}.
\end{equation}
Observe that the numbers $\abs{c^R_\beta}$ can be bounded,
independently
of $R$, by $\int \abs{x^\beta \psi(x)} dx$. This, together with the facts that
$\psi \in \mathcal{S}$ and 
the derivatives of $\eta_R$ are bounded independently of $R$, implies that
there exists a constant $C>0$, independent of $R$, such that
\begin{equation*}
\abs{\partial^\beta(\varphi_R)(x)} \leq C (1+\abs{x})^{-L-1},
\quad \mbox{for all $\abs{\beta} \leq M$ and $x \in \Rdst$}.
\end{equation*}
(The function $\varphi_R$ is compactly supported, but the important point is
that the constant $C$ is independent of $R$).

Consequently,
\begin{equation}
\label{perturb_poly}
\abs{\partial^\beta(\psi-\varphi_R)(x)} \leq K (1+\abs{x})^{-L-1},
\quad \mbox{for all $\abs{\beta} \leq M$ and $x \in \Rdst$},
\end{equation}
with a constant $K>0$ independent of $R$.

Secondly, let us show that for $\abs{\beta} \leq M$,
\begin{equation}
\label{perturb_unif}
\norm{\partial^\beta(\psi-\varphi_R)}_\infty \longrightarrow 0,
\quad \mbox{when } R \longrightarrow +\infty.
\end{equation}
For $\abs{x} < R$, since $\eta_R \equiv 1$ near $x$,
\begin{align*}
 \abs{\partial^\beta(\psi-\varphi_R)(x)}
&\leq \sum_{\abs{\gamma} \leq N} \abs{c^R_\gamma}
\abs{\partial^\beta(\tau_\gamma)(x)}
\\
&\lesssim \sum_{\abs{\gamma} \leq N} \abs{c^R_\gamma}
\\
&=\sum_{\abs{\gamma} \leq N}
\abs{c^R_\gamma - \int_{\Rdst} y^\gamma \psi(y) dy}
\\
&\lesssim \sum_{\abs{\gamma} \leq N} \int_{\abs{y}>R} \abs{y^\gamma
\psi(y)} dy.
\end{align*}
For $\abs{x}\geq R$, Equation \eqref{perturb_poly} gives,
\begin{align*}
 \abs{\partial^\beta(\psi-\varphi_R)(x)} \leq K (1+R)^{-L-1}.
\end{align*}
Hence
\begin{align*}
 \norm{\partial^\beta(\psi-\varphi_R)}_\infty
\lesssim
\max \left\{
\sum_{\abs{\gamma} \leq N} \int_{\abs{y}>R} \abs{y^\gamma \psi(y)} dy,
(1+R)^{-L-1} \right\}
\rightarrow 0,
\end{align*}
as $R\rightarrow +\infty$. Hence, the assertion in Equation
\eqref{perturb_unif} is proved.

Using again the estimate in Equation \eqref{perturb_poly},
we obtain that for $\abs{\beta} \leq N$,
\begin{align*}
\abs{\partial^\beta(\psi-\varphi_R)(x)}
&\leq
\norm{\partial^\beta(\psi-\varphi_R)}_\infty^{1/L+1}
K^{L/L+1}(1+\abs{x})^{-L}.
\end{align*}
Hence, it suffices to choose a value of $R>0$ such that
\begin{align*}
\norm{\partial^\beta(\psi-\varphi_R)}_\infty^{1/L+1}
K^{L/L+1} < \varepsilon, 
\end{align*}
and the claim follows.
\end{proof}

Let us now prove Theorem \ref{th_compact}.
\begin{proof}[Proof of Theorem \ref{th_compact}]
Let $A$ be an expansive matrix and $\Lambda \subseteq \Rdst$ a well-spread
set. Let $\psi \in \wavclass$ (see Remark \ref{rem_class_not_empty}).
Set again $\Psi \supereq \set{\Diltwo{A^j} T_\lambda \psi}
{j \in \Zst, \lambda \in \Lambda}$ and let
$\tilPsi \supereq \set{\tilpsijl}{j \in \Zst, \lambda  \in \Lambda}$
be the dual system from Theorem \ref{th_bandlimited}. Let us further denote
by $\coeftilPsi$ and $\rectilPsi$ the corresponding analysis and
synthesis operators. 

Let us consider the set of parameters $\finiteset := \sett{(p,q,\alpha): p=1,+\infty;
q=1,+\infty; \alpha=\pm \alpha_0}$. For each $(p,q,\alpha) \in \finiteset$,
Theorem \ref{th_bessel_bounds} gives certain constants $L_0, M_0, N_0, C_0$.
Since $\finiteset$ is finite, we can choose them to be the same for all $(p,q,\alpha) \in
\finiteset$. Furthermore, we choose $N_0$ to be greater than the parameter $N$ from the
statement of the theorem.

For every $\varepsilon>0$,
Lemma \ref{lemma_per} yields a smooth, compactly supported
function $\phieps$ with vanishing moments up to order $N$,
such that
\begin{align*}
\set{ 1/\varepsilon \Diltwo{A^j}T_{\lambda} (\psi-\phieps)}
{j\in\Zst, \lambda \in \Lambda},
\end{align*}
is a set of $(L_0,M_0,N_0)$ molecules. Hence, for each $\varepsilon>0$,
\begin{align*}
\norm{\coefPhieps-\coefPsi}_{\Essp \to \essp}
\leq \varepsilon C_0,
\\
\norm{\recPhieps-\recPsi}_{\essp \to \Essp}
\leq \varepsilon C_0,
\end{align*}
where $(p,q,\alpha) \in \finiteset$ and,
\begin{align*}
\Phieps \supereq \set{\phiepsjl := \Diltwo{A^j}T_\lambda \phieps}
{j \in \Zst, \lambda \in \Lambda}.
\end{align*}
According to Theorem \ref{th_bandlimited},
$\recPsi \coeftilPsi = \rectilPsi \coefPsi=I$ on each $\Essp$
(here I denotes the identity operator).
Therefore,
\begin{align*}
\norm{I-\rectilPsi \coefPhieps}_{\Essp \to \Essp}
&=
\norm{\rectilPsi(\coefPsi - \coefPhieps)}_{\Essp \to \Essp}
\\
&\lesssim
\varepsilon \norm{\rectilPsi}_{\essp \to \Essp},
\end{align*}
and,
\begin{align*}
\norm{I - \recPhieps \coeftilPsi}_{\essp \to \essp}
&=
\norm{(\recPsi - \recPhieps)\coeftilPsi}_{\essp \to \essp}
\\
&\lesssim
\varepsilon \norm{\coeftilPsi}_{\essp \to \Essp}.
\end{align*}
Set $\varphi:=\phieps$ with $\varepsilon <<1$ so that the
operators $\rectilPsi \coefPhi$ and $\recPhi \coeftilPsi$
are invertible on $\Essp$ 
for all $(p,q,\alpha) \in \finiteset$. Since the operators
$(\rectilPsi \coefPhi)$, $(\recPhi \coeftilPsi)$,
$(\rectilPsi \coefPhi)^{-1}$ and $(\recPhi \coeftilPsi)^{-1}$
are bounded on $\Essp$ for all $(p,q,\alpha) \in \finiteset$,
Corollary \ref{coro_int} implies that they are bounded on $\Essp$
for all $-\alpha_0 \leq \alpha \leq \alpha_0$,
$1 \leq p,q <+\infty$. For this range of parameters,
a density argument shows that the relations,
\begin{align*}
(\rectilPsi \coefPhi)^{-1} (\rectilPsi \coefPhi) &=
(\rectilPsi \coefPhi) (\rectilPsi \coefPhi)^{-1} = I_{\Essp},
\\
(\recPhi \coeftilPsi)^{-1} (\recPhi \coeftilPsi) &=
(\recPhi \coeftilPsi) (\recPhi \coeftilPsi)^{-1} = I_{\Essp},
\end{align*}
remain valid.

Let us consider the operators
$T := (\rectilPsi \coefPhi)^{-1} \rectilPsi$ and
$T' := \coeftilPsi (\recPhi \coeftilPsi)^{-1}$.

For $j \in \Zst, \lambda \in \Lambda$, let $\delta_{j,\lambda}$
be the sequence taking the value 1 at $(j,\lambda)$ and 0 everywhere else.
Let $\tilphijl := T(\delta_{j,\lambda})$. Since $T$ is bounded 
from $\essp$ to $\Essp$ for all $-\alpha_0 \leq \alpha \leq \alpha_0$, $1
\leq p,q < +\infty$, it follows that $\tilphijl \in \Essp$ for
that range of parameters. Hence, in order to prove (a), it remains
to show that $\tilphijl \in \Essp$ when $p$ or $q$ are
$+\infty$. We already know that $\tilphijl \in \Essp$, 
for all combinations $p=1,+\infty$, $q=1,+\infty$, $\alpha=\pm \alpha_0$,
because $T$ is also bounded from $\essp$ to $\Essp$ for all $(p,q,\alpha) \in
\finiteset$.

The inclusion,
\begin{align*}
\Esp^{-\alpha_0,q}_p \cap \Esp^{\alpha_0,q}_p \subseteq \Essp,
\qquad (-\alpha_0 \leq \alpha \leq \alpha_0),
\end{align*}
is valid for the whole range $1 \leq  p,q \leq +\infty$
and follows easily from the definitions (without
resorting to Theorem \ref{th_interp}). Hence, it suffices to show
that $\tilphijl \in \Essp$ when $\alpha = \pm \alpha_0$ and
$p$ or $q$ are $+\infty$. For $p=+\infty$, this follows from the
inclusion $\Esp^{\pm\alpha_0,1}_\infty \subseteq \Esp^{\pm\alpha_0,q}_\infty$,
for all $1 \leq q \leq +\infty$, which is proved in
\cite[Corollary 3.7]{bo08-3}. Finally, the case $q = +\infty$
follows from the inclusion,
\begin{align*}
\Esp^{\pm\alpha_0,+\infty}_1 \cap \Esp^{\pm\alpha_0,+\infty}_{+\infty}
\subseteq \Esp^{\pm\alpha_0,+\infty}_p,
\qquad (1 \leq p \leq +\infty),
\end{align*}
which in turn follows from \cite[Theorem 1.3]{bo08-3} (and also
from \cite[Theorem 1.1]{bo08-3}).

Let us now establish the expansions of statement (b). To this end,
let $-\alpha_0 \leq \alpha \leq \alpha_0$, $1 \leq p,q <+\infty$.
Since $p,q < \infty$, the set of sequences $\set{\delta_{j,\lambda}}{j \in
\Zst, \lambda \in \Lambda}$ forms an unconditional Schauder basis of
$\essp$. (This follows
for example from the fact that the space $\essp$ is solid and the class of
finitely supported sequences is dense.) Hence, for $f \in \Essp$,
\begin{align*}
\coefPhi(f) = \sum_{j,\lambda} \ip{f}{\phijl} \delta_{j,\lambda},
\end{align*}
with convergence in the norm of $\essp$. Applying $T$ in the last equation
yields the expansion,
\begin{align*}
f = \sum_{j,\lambda} \ip{f}{\phijl} \tilphijl,
\end{align*}
with convergence in the norm of $\Essp$.

Since $\recPhi T' (f) = f$, for all $f \in \Essp$, in order to establish the
dual expansion, it suffices to prove the (formal) adjunction formula,
\begin{align*}
\ip{T'(f)}{\delta_{j,\lambda}} = \ip{f}{T(\delta_{j,\lambda})},
\end{align*}
for all $f \in \Essp$. This is easily seen to hold for $f \in \EsspHil$, where
the pairing is an inner product and the operators $\rectilPsi, \recPhi$ are the
adjoints of $\coeftilPsi, \coefPhi$ respectively.
For $f \in \Essp$ the conclusion follows by a density argument
because of the continuity of the pairing $\ip{\cdot}{\cdot}:\Essp \times
\Esspdual \to \bC$. 

Finally, since $T \coefPhi f = f$ for all $f \in \Essp$, 
it follows that $\norm{f}_\Essp \approx \norm{\coefPhi(f)}_\essp$, for
$-\alpha_0 \leq \alpha \leq \alpha_0$, $1 \leq p,q <+\infty$. Hence
one of the statements in (c) is proved. For the other one, note that we have
just shown that $T'=\coeftilPhi$. Therefore,
$\recPhi \coeftilPhi (f) = \recPhi T' (f)=f$, and $\coeftilPhi: \Essp \to \essp$
is bounded for $-\alpha_0 \leq \alpha \leq \alpha_0$, $1 \leq p,q <+\infty$.
Hence, the conclusion follows.

\end{proof}
As a consequence, we obtain the following corollary which motivated
most of our work.
\begin{coro}
\label{coro_compact_lp}
Let $\Lambda \subseteq \Rdst$ be any well-spread set and let $A \in \Rst^{d
\times d}$ be an expansive
matrix. Then, given $N>0$, there exists a compactly supported, infinitely
differentiable
function $\varphi$, with vanishing moments up to order $N$ such that the
irregular wavelet system
$\set{\Diltwo{A^j}T_{\lambda} \varphi}{j\in\Zst, \lambda \in \Lambda}$
provides
an atomic decomposition for all the spaces $L^p(\Rdst)$, $1 < p < +\infty$
(in the precise sense of Theorem \ref{th_compact}).
\end{coro}
\begin{proof}
The corollary follows immediately from Theorem \ref{th_compact} and the
fact that for $1<p<+\infty$, the unweighted, homogeneous, $A$-anisotropic
Triebel-Lizorkin space $\dot{F}^{0,2}_p$ coincides with $L^p(\Rdst)$ (see
\cite{bo03-1,bo07b}).
\end{proof}
\begin{remark}
Even if one only wanted to prove this corollary, the proof would still go
through the anisotropic
Triebel-Lizorkin spaces $\dot{F}^{0,2}_1$ and $\dot{F}^{0,2}_{\infty}$.
\end{remark}
\begin{remark}
The restriction to the unweighted case $\mu=dx$ is not essential since
all the relevant tools are available in the weighted case
\cite{bo05-1,bo07b,bo08-3}. The required treatment of the
paring $\ip{\cdot}{\cdot}$ is however more technical.
\end{remark}
\begin{remark}[The case $p<1$ or $q<1$]
The perturbation argument from Theorem \ref{th_compact} is based on the fact
that on a Banach space, an operator that is sufficiently close to the
identity is invertible. This is still true in the quasi-Banach case. However,
how close to the identity an operator needs to be in order to be
invertible depends on the constant of the (quasi) triangle inequality. Thus,
the decomposition of Theorem \ref{th_compact} can be extended to include
any individual anisotropic Triebel-Lizorkin space $\Essp$ with $p<1$ or
$q<1$. In contrast to Theorem \ref{th_bandlimited},
the whole range $0 < p,q < +\infty$ cannot be covered with the same window
$\varphi$. 

In connection to this, it should be pointed out that the parameters in the
pairing $\Essp \times \Esspdual \to \bC$ need to be adjusted for $p<1$
or $q<1$
(see \cite[Theorem 4.8]{bo08-3}). Also, for $p \leq 1$, the right spaces
in Corollary \ref{coro_compact_lp} are not $L^p(\Rdst)$ but the anisotropic
Hardy spaces considered in \cite{bo03-1} which do depend on the anisotropy.
\end{remark}

\section{Possible generalization and extensions}
\label{sec:generalization}
We now comment on possible extensions of Theorems \ref{th_bandlimited} and \ref{th_compact}. The proofs in this article can
be readily adapted to cover the use of a different set of translation nodes at each scale,
\begin{align*}
\set{\abs{\det(A)}^{-j/2} \psi(A^{-j} \cdot - \lambda)}
{j \in \Zst, \lambda \in \Lambda_j}, 
\end{align*}
as long as certain qualities of the sets $\Lambda_j$ are kept uniform: the maximum number of points of $\Lambda_j$ in
any cube of side-length 1 and the gap $\rho(\Lambda_j)$.

Another natural question is the possibility of using a different matrix $A_j$ at each scale,
\begin{align*}
\set{\abs{\det(A_j)}^{-j/2} \psi(A_j^{-j} \cdot - \lambda)}
{j \in \Zst, \lambda \in \Lambda_j}.
\end{align*}
This is in principle possible if the family $\sett{A_j: j \in \Zst}$ has an expanding behavior comparable to
the one of the group generated by an expanding matrix $\sett{A^j: j \in \Zst}$. Indeed, the existence of $L^2$-frames
adapted to general families of dilations $\sett{A_j: j \in \Zst}$ has been already studied \cite{wa02-2, alcamo04-1}.
Secondly, the proofs in this article rely on the notion of time-scale molecule which is sufficiently general
so as to cover more general families of dilations (cf. Section \ref{sec:mol}).

The most important pending direction to explore is the computability of the dual frame produced in the proof of Theorem
\ref{th_bandlimited}. It seems that it should be possible to adapt Beurling's methods to produce a computable
\emph{approximate dual frame} for the time-scale systems generated by the generalized painless method, but at this moment
this is not clear from the proof of Theorem \ref{th_bandlimited}. Similarly, a quantitative estimate on the size of the
support of the window produced in Theorem \ref{th_compact} would be of practical relevance.

\section*{Acknowledgement}
We are very grateful to the anonymous referee for her or his valuable comments and suggestions.

\end{document}